\documentclass[12pt,reqno]{amsart}
\usepackage{latexsym}
\usepackage{amssymb}
\usepackage{verbatim}
\usepackage{version}

\catcode`\@=11
\@namedef{subjclassname@2010}{%
  \textup{2010} Mathematics Subject Classification}
\catcode`\@=12

\newtheorem{theorem}{Theorem}
\newtheorem{proposition}[theorem]{Proposition}

\newtheorem{lemma}[theorem]{Lemma}
\newtheorem{corollary}[theorem]{Corollary}

\theoremstyle{remark}

\newtheorem{remark}[theorem]{Remark}
\newtheorem{remarks}[theorem]{Remarks}

\numberwithin{equation}{section}

\def\PSL{\operatorname{PSL}}
\newcommand{\Z}{\mathbb{Z}}

\newcommand{\R}{\mathbb{\R}}

\includeversion{comment}
\begin{document}

\title[Periodicity of free subgroup numbers modulo prime powers]{Periodicity of free subgroup numbers modulo prime powers} 
\author[C. Krattenthaler and 
T.\,W. M\"uller]{C. Krattenthaler$^{\dagger}$ and
T. W. M\"uller$^*$} 

\address{$^{\dagger}$Fakult\"at f\"ur Mathematik, Universit\"at Wien, 
Oskar-Morgenstern-Platz~1, A-1090 Vienna, Austria.
WWW: {\tt http://www.mat.univie.ac.at/\lower0.5ex\hbox{\~{}}kratt}.}

\address{$^*$School of Mathematical Sciences, Queen Mary
\& Westfield College, University of London,
Mile End Road, London E1 4NS, United Kingdom.}

\thanks{$^\dagger$Research partially supported by the Austrian
Science Foundation FWF, grants Z130-N13 and S50-N15,
the latter in the framework of the Special Research Program
``Algorithmic and Enumerative Combinatorics"}



\begin{abstract}
We completely characterise when the sequence of free subgroup numbers of a finitely generated virtually free group is ultimately periodic modulo a given prime power.
\end{abstract}

\maketitle

\section{Introduction}
\label{sec:intro}

\noindent For a finitely generated virtually free group $\Gamma$, denote by
$m_\Gamma$ the least common multiple of the orders of the finite
subgroups in $\Gamma$ and, for a 
positive integer $\lambda$, let $f_\lambda(\Gamma)$ denote the number
of free subgroups of index $\lambda m_\Gamma$ in $\Gamma$.  
In \cite{KM2}, the authors show, among other things, that the number
$f_\lambda(PSL_2(\Z))$ of free subgroups 
of index~$6\lambda$ in the inhomogeneous modular group $\PSL_2(\Z)$,
considered as a sequence indexed by $\lambda$, is 
ultimately periodic modulo
any fixed prime power $p^\alpha$, if $p$ is a prime number with $p\geq5$. 
More precise results on the length of the period, and an explicit
formula for the linear recurrence satisfied by these numbers 
modulo~$p^\alpha$ are also provided in \cite{KM2}. As is well known,
ultimate periodicity of the sequence
$\big(f_\lambda(\Gamma)\big)_{\lambda\geq1}$ is equivalent to rationality of
the corresponding generating function $F_\Gamma(z) =
\sum_{\lambda\geq0} f_{\lambda+1}(\Gamma) z^\lambda$.  

The purpose of the present paper is to demonstrate that the 
periodicity phenomenon discovered in \cite{KM2} holds in a much wider
context, namely that of finitely generated virtually free
groups. Indeed, our main result (Theorem~\ref{Thm:Main}) provides an
explicit characterisation of all pairs $(\Gamma, p^\alpha)$, where
$\Gamma$ is a finitely generated virtually free group and $p^\alpha$
is a proper prime power, for which the sequence of free subgroup
numbers of $\Gamma$ is ultimately periodic modulo~$p^\alpha$.  
Roughly speaking, for ``almost all" pairs $(\Gamma, p)$ the
sequence $\big(f_\lambda(\Gamma)\big)_{\lambda\geq0}$ is ultimately periodic
modulo~$p^\alpha$ for all $\alpha\geq1$, the only exception
occurring when $p\mid m_\Gamma$ and $\mu_p(\Gamma)=0$, where
$\mu_p(\Gamma)$ is a certain invariant defined in 
(\ref{Eq:mupDef}) and discussed in the paragraph following that formula. 

In order to further place our results into context, we point out that,
for primes $p$ dividing the constant $m_\Gamma$, an elaborate theory
is presented in \cite{MuPuFree} for the behaviour of the
arithmetic function $f_\lambda(\Gamma)$ modulo~$p$.  Recently, this theory has
been supplemented by congruences modulo (essentially arbitrary)
$2$-powers and $3$-powers for the number of free subgroups of finite
index in lifts of the classical modular group; that is, amalgamated
products of the form 
\begin{equation*}
\Gamma_\ell = C_{2\ell} \underset{C_\ell}{\ast} C_{3\ell},\quad \ell\geq1;
\end{equation*}
cf.\ Theorems~19 and 20 in \cite[Sec.~8]{KKM}, and  
Section~16 in \cite{KM}, in particular, \cite[Thms.~49--52]{KM}. 
These results demonstrate a highly
non-trivial behaviour of the sequences 
$\big(f_\lambda(\Gamma_\ell)\big)_{\lambda\ge1}$ modulo
powers of $2$ if $\ell$ is odd (in which case $\mu_2(\Gamma_\ell)=0$), 
and modulo powers of $3$ for $3\nmid \ell$ (in which case 
$\mu_3(\Gamma_\ell)=0$). For instance, for the sequence
$\big(f_\lambda=f_\lambda(\Gamma_1)\big)_{\lambda\ge1}$ of  
free subgroup numbers of the group $\PSL_2(\Z)$, one obtains that:
\begin{enumerate} 
\item $f_\lambda\equiv-1~(\text{mod }3)$ if, and only if, 
the $3$-adic expansion of $\lambda$ is an element of
$\{0,2\}^*1$;
\vspace{2mm}

\item $f_\lambda\equiv1~(\text{mod }3)$ if, and only if, 
the $3$-adic expansion of $\lambda$ is an element of
$$
\{0,2\}^*100^*\cup \{0,2\}^*122^*;$$

\item for all other $\lambda$, we have $f_\lambda\equiv0~(\text{mod }3)$;
\end{enumerate}
cf.\ \cite[Cor.~53]{KM}. Here, for a set $\Omega$, we denote by
$\Omega^\ast$ the free monoid generated by $\Omega$. 
All this is in sharp contrast to ``most" of the cases in
the classification result in Theorem~\ref{Thm:Main},
which exhibit ``simple" (ultimate) periodicity.

In proving Theorem~\ref{Thm:Main}, the bulk of the argument lies in
showing that, if $p$ is a prime number {\it not\/} dividing $m_\Gamma$, then
the sequence $\big(f_\lambda(\Gamma)\big)_{\lambda\geq1}$ is ultimately
periodic modulo~$p^\alpha$ for  every integer $\alpha\geq1$; this is
the contents of Theorem~\ref{Thm:Periodic}, whose proof occupies
Sections~\ref{Sec:Periodic}--\ref{Sec:PeriodicThmProof}. 
 The case where $p\mid m_\Gamma$ is largely taken care of by
Theorem~\ref{Thm:mup>0}; its proof in
Section~\ref{Sec:prop} is by an inductive argument, which is
based on an earlier generating function result in \cite{MuPuFree}. The proof of Theorem~\ref{Thm:Main} itself appears in Section~\ref{Sec:Main}.
Precise formulations of our  results are found in Section~\ref{Sec:Results},
while the next section collects together definitions as well as some background
material on virtually free groups.

\section{Some preliminaries on virtually free groups}
\label{Sec:Prelim}

\noindent Our  notation and terminology here follows Serre's book \cite{Serre2};
in particular, the category of graphs used is described in
\cite[\S2]{Serre2}. 
This category deviates slightly from the usual notions in graph theory.
Specifically, a {\it graph} $X$ consists of two sets:
$E(X)$, the set of (directed) {\it edges}, and $V(X)$, the set of
{\it vertices}. The set $E(X)$ is endowed with a fixed-point-free involution
${}^-: E(X) \rightarrow E(X)$ ({\it reversal of orientation}), and there are
two functions $o,t: E(X)\rightarrow V(X)$ assigning to an edge $e\in
E(X)$ its {\it origin} $o(e)$ and {\it terminus} $t(e)$, such that $t(\bar{e}) =
o(e)$. 
The reader should note that, according to the above
definition, graphs may have loops (that is, edges $e$ with
$o(e)=t(e)$) and multiple edges (that is, several edges with
the same origin and the same terminus).  
An {\it orientation} $\mathcal{O}(X)$ consists of a choice of exactly 
one edge in each pair $\{e, \bar{e}\}$ 
(this is indeed always a {\it pair} -- even for loops --
since, by definition, the involution
${}^-$ is fixed-point-free). Such a pair is called a
{\it geometric edge}. 

Let $\Gamma$ be a finitely generated virtually free group with
Stallings decomposition\break
 $(\Gamma(-), X)$; that is, $(\Gamma(-), X)$ is
a finite graph of finite groups with fundamental group $\Gamma \cong
\pi_1(\Gamma(-), X)$. If $\mathfrak{F}$ is a free subgroup of finite
index in $\Gamma$ then, following an idea of C.\,T.\,C. Wall, one
defines the (rational) Euler characteristic $\chi(\Gamma)$ of $\Gamma$
as 
\begin{equation}
\label{Eq:EulerChar}
\chi(\Gamma) = - \frac{\mathrm{rk}(\mathfrak{F}) - 1}{(\Gamma:\mathfrak{F})}.
\end{equation}
(This is well-defined in view of Schreier's index formula  in
\cite{Schreier}.)  
In terms of the above decomposition of $\Gamma$, we have
\begin{equation}
\label{Eq:EulerCharDecomp}
\chi(\Gamma) = \sum_{v\in V(X)} \frac{1}{\vert
\Gamma(v)\vert}\,-\,\sum_{e\in\mathcal{O}(X)} \frac{1}{\vert
\Gamma(e)\vert}. 
\end{equation}
Equation~\eqref{Eq:EulerCharDecomp} reflects the fact that, in our
situation, the Euler characteristic in the sense of Wall coincides
with the equivariant Euler characteristic $\chi_T(\Gamma)$ of $\Gamma$
relative to the tree $T$ canonically associated with $\Gamma$ in the
sense of Bass--Serre theory; cf.\ \cite[Chap.~IX, Prop.~7.3]{Brown} or
\cite[Prop.~14]{Serre1}. We remark that a  finitely generated
virtually free group $\Gamma$ is largest among finitely generated
groups in the sense of Pride's preorder \cite{Pride} (i.e., $\Gamma$
has a subgroup of finite index, which can be mapped onto the free
group of rank $2$) if, and only if, $\chi(\Gamma)<0$; see
Lemma~\ref{Lem:chi<0Equivs} in Section~\ref{Sec:Large}.   

As in the introduction, denote by $m_\Gamma$ the least common multiple
of the orders of the finite subgroups in $\Gamma$, so that, again in
terms of the above Stallings decomposition of $\Gamma$,  
\[
m_\Gamma = \mathrm{lcm}\big\{\vert\Gamma(v)\vert:\, v\in V(X)\big\}.
\]
(This formula essentially follows from the well-known fact that a
finite group has a fixed point when acting on a tree.)  
The type $\tau(\Gamma)$ of a finitely generated virtually free group
$\Gamma \cong \pi_1(\Gamma(-), X)$ is defined as the  
tuple
\[
\tau(\Gamma) = \big(m_\Gamma; \zeta_1(\Gamma), \ldots, \zeta_\kappa(\Gamma), \ldots, \zeta_{m_\Gamma}(\Gamma)\big),
\]
where the  
$\zeta_\kappa(\Gamma)$'s are integers indexed by the divisors of
$m_\Gamma$, given by 
\[
\zeta_\kappa(\Gamma) = 
\big\vert\big\{e\in \mathcal{O}(X):\, \vert\Gamma(e)\vert
\,\big\vert\, \kappa\big\}\big\vert\,-\, \big\vert\big\{v\in V(X):\,
\vert\Gamma(v)\vert \,\big\vert\, \kappa\big\}\big\vert. 
\]
It can be shown that the type $\tau(\Gamma)$ is in fact an invariant
of the group $\Gamma$, i.e., independent of the particular
decomposition of $\Gamma$ in terms of a graph of groups $(\Gamma(-),
X)$, and that two finitely generated virtually free groups $\Gamma_1$
and $\Gamma_2$ contain the same number of free subgroups of index $n$
for each positive integer $n$ if, and only if, $\tau(\Gamma_1) =
\tau(\Gamma_2)$; cf.\ \cite[Theorem~2]{MuDiss}. We have
$\zeta_\kappa(\Gamma)\geq0$ for $\kappa<m_\Gamma$ and
$\zeta_{m_\Gamma}(\Gamma)\geq-1$ with equality occurring in the latter
inequality if, and only if, $\Gamma$ is the fundamental group of a
tree of groups; cf.\ \cite[Prop.~1]{MuEJC} or \cite[Lemma~2]{MuDiss}. 

We observe that, as a
consequence of (\ref{Eq:EulerCharDecomp}), the Euler characteristic of
$\Gamma$ can be expressed in terms of the type $\tau(\Gamma)$ via 
\begin{equation}
\label{Eq:EulerTypeRewrite}
\chi(\Gamma) = - m_\Gamma^{-1} \sum_{\kappa \mid m_\Gamma}
\varphi(m_\Gamma/\kappa) \,\zeta_\kappa(\Gamma), 
\end{equation}
where $\varphi$ is Euler's totient function. It follows in particular 
that, if two finitely generated virtually free groups have the same
number of free subgroups of index $n$ for every $n$, then their Euler
characteristics must coincide. 

The proof of Theorem~\ref{Thm:Periodic}, as given in
Section~\ref{Sec:PeriodicThmProof}, is based on the analysis of a second
arithmetic function associated with the group $\Gamma$. Define a
\textit{torsion-free $\Gamma$-action on a set $\Omega$} to be a
$\Gamma$-action on $\Omega$ which is free when restricted to finite
subgroups, and let 
\[
g_\lambda(\Gamma):= \frac{\mbox{number of torsion-free
$\Gamma$-actions on a set with $\lambda m_\Gamma$ elements}}{(\lambda
m_\Gamma)!},\quad \lambda\geq0; 
\]
in particular, $g_0(\Gamma)=1$. The sequences
$\big(f_\lambda(\Gamma)\big)_{\lambda\geq1}$ and
$\big(g_\lambda(\Gamma)\big)_{\lambda\geq0}$ are related via the
Hall-type transformation formula\footnote{See \cite[Cor.~1]{MuDiss}, or
\cite[Prop.~1]{DM} for a more general result.}  
\begin{equation}
\label{Eq:Transform}
\sum_{\mu=0}^{\lambda-1} g_\mu(\Gamma) f_{\lambda-\mu}(\Gamma) =
m_\Gamma \lambda g_\lambda(\Gamma),\quad \lambda\geq1. 
\end{equation}
Moreover, a careful analysis of the universal mapping property
associated with the presentation $\Gamma\cong\pi_1(\Gamma(-), X)$
leads to the explicit formula 
\begin{equation}
\label{Eq:gExplicit}
g_\lambda(\Gamma) = \frac{\prod\limits_{e\in\mathcal{O}(X)} (\lambda
m_\Gamma/\vert\Gamma(e)\vert)!\,\vert\Gamma(e)\vert^{\lambda
m_\Gamma/\vert\Gamma(e)\vert}}{\prod\limits_{v\in V(X)}(\lambda
m_\Gamma/\vert\Gamma(v)\vert)!\,\vert\Gamma(v)\vert^{\lambda
m_\Gamma/\vert\Gamma(v)\vert}},\quad \lambda\geq0 ,
\end{equation}
for $g_\lambda(\Gamma)$, where $\mathcal{O}(X)$ is any orientation of
$X$; cf.\ \cite[Prop.~3]{MuDiss}. Introducing the generating functions 
\[
F_\Gamma(z) := \sum_{\lambda\geq0} f_{\lambda+1}(\Gamma) z^\lambda
\,\mbox{ and }\, G_\Gamma(z):= \sum_{\lambda\geq0} g_\lambda(\Gamma)
z^\lambda, 
\]
Equation~\eqref{Eq:Transform} is seen to be equivalent to the relation
\begin{equation}
\label{Eq:GFTransform}
F_\Gamma(z) = m_\Gamma \frac{d}{dz}\big(\log G_\Gamma(z)\big).
\end{equation}

Define the \textit{free rank} $\mu(\Gamma)$ of a finitely generated
virtually free group $\Gamma$ to be the rank of a free subgroup of
index $m_\Gamma$ in $\Gamma$ (existence of such a subgroup follows,
for instance, from Lemmas~8 and 10 in \cite{Serre2}; it need not be 
unique, though). We note that, in view of \eqref{Eq:EulerChar}, the
quantity $\mu(\Gamma)$ is connected with the Euler characteristic of
$\Gamma$ via  
\begin{equation}
\label{Eq:FreeEuler}
\mu(\Gamma) + m_\Gamma \chi(\Gamma) = 1,
\end{equation}
which shows in particular that $\mu(\Gamma)$ is
well-defined. Combining Equations (\ref{Eq:EulerTypeRewrite}) and
(\ref{Eq:FreeEuler}), we see that the free rank $\mu(\Gamma)$ can be
expressed in terms of the type of $\Gamma$ via  
\begin{equation}
\label{Eq:muTypeRewrite}
\mu(\Gamma) = 1 + \sum_{\kappa\mid m_\Gamma} \varphi(m_\Gamma)/\kappa) \zeta_\kappa(\Gamma).
\end{equation}
Given a finitely generated virtually free group $\Gamma$ and a prime number $p$, we introduce, in analogy with formula (\ref{Eq:muTypeRewrite}), the $p$-rank $\mu_p(\Gamma)$ of $\Gamma$ via the equation 
\begin{equation}
\label{Eq:mupDef}
\mu_p(\Gamma) = 1 + \sum_{p\mid \kappa \mid m_\Gamma} \varphi(m_\Gamma)/\kappa) \zeta_\kappa(\Gamma).
\end{equation}
Clearly, $\mu_p(\Gamma)\geq0$, with equality occurring in this
inequality if, and only if, $\Gamma$ is the fundamental group of a
tree of groups and $\zeta_\kappa(\Gamma)=0$ for $p\mid \kappa\mid
m_\Gamma$ and $\kappa<m_\Gamma$. Similarly, we have $\mu_p(\Gamma)=1$
if, and only if, (i) $\zeta_\kappa(\Gamma)=0$ for all $\kappa$ with
$p\mid \kappa\mid m_\Gamma$, or (ii) $\Gamma$ is the fundamental group
of a tree of groups, $m_\Gamma$ is even, $p\mid m_\Gamma/2$,
$\zeta_{m_\Gamma/2}(\Gamma)=1$, and $\zeta_\kappa(\Gamma)=0$ for
$p\mid \kappa\mid m_\Gamma$ and $\kappa<m_\Gamma/2$. 
To give concrete examples,
if $p$ is an odd prime number, then the groups 
\[
\Gamma_{p, \alpha} = C_2 \ast C_{2p} \ast \underbrace{C_p \ast \cdots \ast C_p}_{\alpha\,\mathrm{copies}},\quad \alpha\geq0
\]
satisfy $\mu_p(\Gamma_{p, \alpha}) = 1$, while the groups
\[
\Gamma_{2, \alpha} = C_4\ast C_4\ast \underbrace{C_2\ast \cdots C_2}_{\alpha\,\mathrm{copies}},\quad \alpha\geq0
\]
satisfy $\mu_2(\Gamma)=1$.

\section{The results}
\label{Sec:Results}
\noindent Here and in the sequel,
given power series $f(z)$ and $g(z)$,
we write 
$$f(z)=g(z)~\text {modulo}~p^\gamma$$ 
to mean that the coefficients
of $z^i$ in $f(z)$ and $g(z)$ agree modulo~$p^\gamma$ for all $i$; in
particular, the phrase ``$F_\Gamma(z)$ is rational modulo~$p^\alpha$''
means that $F_\Gamma(z)$ equals a certain rational function
modulo~$p^\alpha$ in the sense of the above definition. Our main
result, which completely characterises rationality of the generating
function $F_\Gamma(z)$ (i.e., ultimate periodicity of the sequence
$\big(f_\lambda(\Gamma)\big)_{\lambda\ge1}$) 
modulo prime powers, is as follows. 

\begin{theorem}
\label{Thm:Main}
Let $\Gamma$ be a finitely generated virtually free group, let $p$ be
a prime number, and let $F_\Gamma(z)$ denote the generating
function $\sum_{\lambda\geq0} f_{\lambda+1}(\Gamma) z^\lambda$
for the free subgroup numbers of $\Gamma$.
Then the following assertions are equivalent:

\begin{enumerate} 
\item [(i)] the series $F_\Gamma(z)$ is rational
modulo~$p^\alpha$ for each positive integer $\alpha;$
\vspace{2mm}
\item [(ii)] the series $F_\Gamma(z)$ is rational
modulo $p;$
\vspace{2mm}
\item [(iii)] 
The pair $(\Gamma,p)$ satisfies one of the following
mutually exclusive conditions:
\vspace{1mm}
\begin{enumerate} 
\item [(iii)$_1$] $p\nmid m_\Gamma;$
\vspace{1mm}
\item [(iii)$_2$] $p\mid m_\Gamma$ and $\mu_p(\Gamma)>0;$
\vspace{1mm}
\item [(iii)$_3$]
$\Gamma$ is finite;
\vspace{1mm}
\item [(iii)$_4$]
$\Gamma$ is virtually infinite-cyclic and $p = 2$.
\end{enumerate}
\end{enumerate}
\end{theorem}
The proof of Theorem~\ref{Thm:Main} is broken up into two main steps,
each of which is a meaningful result in its own right. Case (iii)$_1$
is taken care of by the following fundamental result. 
\begin{theorem}
\label{Thm:Periodic}
Let $\Gamma$ be a finitely generated virtually free group, let $p$ be
a prime number not dividing $m_\Gamma,$ and let $\alpha$ be a positive
integer. Then the sequence $\big(f_\lambda(\Gamma)\big)_{\lambda\ge1}$ 
is ultimately periodic modulo~$p^\alpha$. 
\end{theorem}

The case (iii)$_2$, where $p\mid m_\Gamma$ and $\mu_p(\Gamma)>0$, is dealt
with in our last result. 

\begin{theorem}
\label{Thm:mup>0}
Let $\Gamma$ be a finitely generated virtually free group, let $p$ be
a prime number such that $p\mid m_\Gamma$ and $\mu_p(\Gamma)>0,$
and let $\alpha$ be a positive
integer. Then the generating function $F_\Gamma(z)$ for the free subgroup
numbers of $\Gamma$ is rational modulo~$p^\alpha$. 
If $\mu_p(\Gamma)=1,$ then $F_\Gamma(z)$ is non-polynomial 
modulo~$p^\alpha$ and the rational fraction
can be written so that the denominator is a power of\/ $1-z$ or $1+z$,
depending on whether $(\mu(\Gamma)-\mu_p(\Gamma))/(p-1)$ is even or odd.
If $\mu_p(\Gamma)\ge2,$ then $F_\Gamma(z)$
is polynomial modulo~$p^\alpha$. 
\end{theorem}

Other ingredients in the proof of Theorem~\ref{Thm:Main} are Corollary~\ref{Cor:mu=1}, which describes the function $f_\lambda(\Gamma)$ in the case where $\Gamma$ is virtually infinite-cyclic, as well as \cite[Prop.~2]{
MuPuFree} and \cite[Theorem~2]{MuPuFree}. 

The proof of Theorem~\ref{Thm:Periodic}, as given in
Section~\ref{Sec:PeriodicThmProof}, is based on the analysis of a second,
rational-valued,  
arithmetic function associated with the group $\Gamma$,
whose $\lambda$-th term may be viewed as the ``normalised"
number of torsion-free $\Gamma$-actions on a set with $\lambda m_\Gamma$ 
elements. This function, 
denoted here by $g_\lambda(\Gamma)$, is connected with the free
subgroup numbers $f_\lambda(\Gamma)$ via the crucial equation
\eqref{Eq:Transform}; see Section~\ref{Sec:Prelim}. A careful $p$-adic
analysis of this equation 
will show in the end that the numbers $f_\lambda(\Gamma)$ satisfy a
linear recurrence modulo any fixed prime power $p^\alpha$ with $p$
not dividing $m_\Gamma$. By standard results on linear recurring
sequences, Theorem~\ref{Thm:Periodic} then follows immediately.

The function $g_\lambda(\Gamma)$ has been discussed in
Section~\ref{Sec:Prelim}, 
together with some further preliminaries on
virtually-free groups. 
Sections~\ref{Sec:Periodic}--\ref{sec:mainlemma} prepare for the
proof of Theorem~\ref{Thm:Periodic} in Section~\ref{Sec:PeriodicThmProof}: 
Section~\ref{Sec:Periodic} surveys the relevant results from the theory
of linear recurring
sequences; Section~\ref{Sec:GoGNorm} provides a representation of a
finitely generated virtually free group in terms of a finite graph of finite
groups without trivial amalgamations, a 
representation which is
particularly suited 
for our subsequent analysis; Section~\ref{Sec:TreeOrient} contains
a purely graph-theoretic lemma on orientation of trees; 
Section~\ref{sec:class} discusses the classification
of virtually free groups of (free) rank at most $2$; Section~\ref{Sec:Large}
collects together criteria for finitely generated virtually free groups to be
`large', while Section~\ref{sec:mainlemma} establishes a crucial 
$p$-divisibility property for this second arithmetic function
$g_\lambda(\Gamma)$.

\medskip
We conclude this section with some remarks.
 
\begin{remarks}
(1) It is shown in \cite{MuPuFree} that, if $\Gamma$ is a finitely
  generated virtually free group with $\mu_p(\Gamma)=0$ for a given
  prime~$p$, then the function $f_\lambda(\Gamma)$ satisfies the
  congruence 
\[
f_\lambda(\Gamma) \equiv (-1)^{\frac{(\mu(\Gamma)-1)}{p-1}}
\lambda^{-1}
\binom{\frac{\mu(\Gamma)\lambda}{p-1}}{\frac{\lambda-1}{p-1}}
\quad \text{mod } p; 
\]
cf.\ \cite[Eqn.~(35)]{MuPuFree}. In general, it remains an open
problem, how the free subgroup numbers of a finitely generated
virtually free group $\Gamma$ with $\mu_p(\Gamma)=0$ behave modulo higher
$p$-powers. The only results known in this direction concern (i) lifts
of Hecke groups $\mathfrak{H}(q) \cong C_2 \ast C_q$ with $q$ a Fermat
prime and $p=2$, and (ii) lifts of the classical modular group
$\mathfrak{H}(3)\cong PSL_2(\Z)$ with $p=3$; see Corollary~34 and Theorem~35 in
\cite{KKM}, and \cite[Sec.~16]{KM}.   

\medskip
(2)  By a \textit{cyclic cover}, we mean the fundamental group $\Gamma$ of
a finite graph $(\Gamma(-), X)$ of finite cyclic groups. To fix ideas,
we shall assume that the canonical embeddings associated with
$(\Gamma(-), X)$ are induced by the identity maps of the corresponding
vertex stabilisers. Let $\Gamma = \pi_1(\Gamma(-), X)$ be a cyclic
cover, and let $\ell$ be  a positive integer. Then we define the
\textit{$\ell$-th lift $\Gamma_\ell$  
of\/ $\Gamma$} as the cyclic cover resulting from $(\Gamma(-), X)$ by multiplying the order of each (vertex or edge) stabiliser by a factor $\ell$. The last assertion in Theorem~\ref{Thm:mup>0} implies that $F_{\Gamma_\ell}(z)$ is polynomial modulo all proper $p$-powers for all lifts $\Gamma_\ell$ of cyclic covers $\Gamma$ with $p \nmid m_\Gamma$, $\mu(\Gamma)\geq2$, and $p\mid \ell$.

\medskip
 (3) In order to illustrate Theorem~\ref{Thm:mup>0}, let us consider
the case where $\Gamma=\mathfrak H(6)\cong C_2\ast C_6$ and $p=3$. 
Indeed, in this example, we have $3\mid m_{\mathfrak H(6)}=6$ and
$\mu_3(\mathfrak H(6))=1$. If one applies the algorithm which is
implicit in the proof of Theorem~\ref{Thm:mup>0} given in
Section~\ref{Sec:prop}, then, modulo~$3^9=19683$, one obtains
\begin{multline*}
F_{\mathfrak H(6)}(z)=
\frac {1} {(1+z)^{10}}
(19680 z^9+585 z^8+1926
   z^7+6165 z^6+7326 z^5\\+1584
   z^4+1566 z^3+17433 z^2+1845
   z+15)
\quad \quad \text{modulo }3^9.
\end{multline*}
Consequently, the period length of the ultimately periodic sequence
$\big(f_\lambda(\mathfrak H(6))\big)_{\lambda\ge1}$ 
is $2\cdot 3^{11}=354294$ when taken modulo~$3^9$.

\medskip
(4) For a finitely generated virtually free group $\Gamma$, denote by
$\hat{\Gamma}$ the isomorphism class of $\Gamma$. Given a prime number
$p$, define a density $\mathfrak{D}_p$ of isomorphism classes
of groups $\Gamma$ with $\mu_p(\Gamma)=0$ in all isomorphism classes by  
\[
\mathfrak{D}_p:= \lim_{M\rightarrow\infty}
\frac{\big\vert\big\{\hat{\Gamma}:\,m_\Gamma\leq M,\, \mu(\Gamma)\leq
  M,\,
  \mu_p(\Gamma)=0\big\}\big\vert}{\big\vert\big\{\hat{\Gamma}:\,m_\Gamma\leq
  M,\, \mu(\Gamma)\leq N\big\}\big\vert}. 
\]
Note that this definition makes sense in view of \cite[Prop.~4]{MuDiss} and Equation~(\ref{Eq:FreeEuler}). We conjecture that $\mathfrak{D}_p=0$ for all prime numbers $p$.

\end{remarks}

\section{Periodicity of sequences over finite rings}
\label{Sec:Periodic}

\noindent In this section we review some standard results on linear
recurring sequences (usually only formulated over finite fields), which will be used in a crucial manner 
in the proof of Theorem~\ref{Thm:Periodic} in
Section~\ref{Sec:PeriodicThmProof}.

Let $\Lambda$ be a finite commutative ring with identity, and let
$\mathcal{S} = (s_n)_{n\geq0}$ be a sequence of elements of
$\Lambda$. Suppose that there exist a positive integer $d$ and elements
$\alpha_0$, $\alpha_1, \ldots, \alpha_{d-1}, \alpha\in\Lambda$, such
that $\mathcal{S}$ satisfies the relation 
\begin{equation}
\label{Eq:LinRec}
s_{n+d} = \alpha_{d-1} s_{n+d-1} + \alpha_{d-2} s_{n+d-2} + \cdots +
\alpha_0 s_n + \alpha,\quad n\geq0. 
\end{equation}
Then $\mathcal{S}$ is termed a \textit{linear recurring sequence} over
$\Lambda$ \textit{of order $d$}, a relation of the form
\eqref{Eq:LinRec} itself is called a \textit{linear recurrence
relation} (or difference relation) \textit{of order
$d$}. Relation~\eqref{Eq:LinRec} is called \textit{homogeneous} if
$\alpha=0$, otherwise \textit{inhomogeneous}; the sequence
$\mathcal{S}$ itself is called a \textit{homogeneous}, or
\textit{inhomogeneous},  \textit{linear recurring sequence} over
$\Lambda$, respectively. 

The sequence $\mathcal{S} = (s_n)_{n\geq0}$ is termed
\textit{ultimately periodic}, if there exist integers $\omega>0$ and
$n_0\geq0$, such that $s_{n+\omega} = s_n$ holds for all $n\geq
n_0$. The integer $\omega$ is then called a \textit{period} of
$\mathcal{S}$. The smallest number among all the possible periods
$\omega$ of an ultimately periodic sequence $\mathcal{S}$ is called
the \textit{least period} $\omega_0=\omega_0(\mathcal{S})$ of
$\mathcal{S}$. If $\mathcal{S} = (s_n)_{n\geq0}$ is ultimately
periodic with least period $\omega_0$, then the least non-negative
integer $n_0$ such that $s_{n+\omega_0}=s_n$ for all $n\geq n_0$ is
called the \textit{preperiod} of $\mathcal{S}$. An ultimately periodic
sequence $\mathcal{S} = (s_n)_{n\geq0}$ with least period
$\omega_0(\mathcal{S})$ is termed \textit{purely periodic}, if
$s_{n+\omega_0(S)}=s_n$ for all $n\geq0$. It is easy to see that a
sequence $\mathcal{S} = (s_n)_{n\geq0}$ is purely periodic, if, and
only if, there exists an integer $\omega>0$ such that
$s_{n+\omega}=s_n$ for all $n\geq0$. Also, every period of an
ultimately periodic sequence is divisible by the least period.  

Linear recurring sequences over finite rings are always (ultimately)
periodic. The following result, which concentrates on the case of a
homogeneous linear recurring sequences, will suffice for our present
purposes.  

\begin{lemma}
\label{Lem:Periodicity}
Let $\mathcal{S} = (s_n)_{n\geq0}$ be a homogeneous linear recurring
sequence of order $d\geq1$ over a finite commutative ring $\Lambda$
with identity. Then $\mathcal{S}$ is ultimately periodic with least
period $\omega_0(\mathcal{S}) < \vert \Lambda\vert^d$. Moreover, if
the linear recurrence relation \eqref{Eq:LinRec} satisfied by
$\mathcal{S}$ is such that $\alpha_0$ is invertible {\em(}i.e., a unit of
$\Lambda${\em)}, then $\mathcal{S}$ is purely periodic.  
\end{lemma}

The proof of Lemma~\ref{Lem:Periodicity} is virtually identical with
that in the case of finite fields; see Chapter~8 in \cite{LiNie}, in
particular Theorems~8.7 and 8.11, and Lemma~8.12.

\section{Normalising a finite graph of groups}
\label{Sec:GoGNorm}
\noindent It will be helpful to be able to represent a finitely generated
virtually free group $\Gamma$ by a graph of groups avoiding trivial
amalgamations. This is achieved via the following. 

\begin{lemma}[\sc Normalisation]
\label{Lem:Normalise}
Let $(\Gamma(-), X)$ be a {\em(}connected{\em)} graph of groups with fundamental
group $\Gamma,$ and suppose that $X$ has only finitely many
vertices. Then there exists a graph of groups $(\Delta(-), Y)$ with
$\vert V(Y)\vert < \infty$ and a spanning tree $T$ in $Y,$ such that
$\pi_1(\Delta(-),Y) \cong \Gamma$, and such that\footnote{The notation 
used in Equation~\eqref{eq:normal} follows Serre; 
see D\'ef.~8 in \cite[Sec.~4.4]{Serre2}.} 
\begin{equation} \label{eq:normal}
\Delta(e)^e \neq \Delta(t(e))\,\mbox{ and }\, \Delta(e)^{\bar{e}} \neq
\Delta(o(e)),\quad 
\text {for }e\in E(T). 
\end{equation}
Moreover, if $(\Gamma(-), X)$ satisfies the finiteness condition
\bigskip

\centerline{\hskip3cm
$(F_1)$ \qquad \mbox{$X$ is a finite graph,}\hfill}
\bigskip

\noindent
or
\bigskip

\centerline{\hskip3cm $(F_2)$  \qquad 
\mbox{$\Gamma(v)$ is finite for every vertex $v\in V(X),$}\hfill}
\bigskip

\noindent
then we may choose $(\Delta(-), Y)$ so as to enjoy the same property.
\end{lemma}
\begin{proof}
Choose a 
spanning tree $S$ in $X$, and call an edge $e\in E(S)$
\textit{trivial}, if at least one of the associated embeddings $e:
\Gamma(e) \rightarrow \Gamma(t(e))$ and $\bar{e}: \Gamma(e)
\rightarrow \Gamma(o(e))$ is an isomorphism. If $S$ contains a trivial
edge $e_1$, to fix ideas, say $\Gamma(e_1)^{e_1} = \Gamma(t(e_1))$,
then we contract the edge $e_1$ into the vertex $o(e_1)$ and re-define
incidence and embeddings where necessary, to obtain a new graph of
groups $(\Gamma'(-), X')$ with spanning tree $S'$ in $X'$. More
precisely, this means that we let  
\begin{align*}
E(X') &= E(X) \setminus \{e_1, \bar{e}_1\},\\[1mm]
E(S') &= E(S) \setminus \{e_1, \bar{e}_1\},\\[1mm]
V(X') &= V(S') = V(X)\setminus \{t(e_1)\},
\end{align*}
set 
\[
t'(e):= o(e_1), \quad \text {for }e\in E(X')\text { with }t(e) = t(e_1),
\] 
and define new embeddings via
\begin{equation}
\label{Eq:ImbedExt}
\Gamma(e) \overset{e}{\longrightarrow} \Gamma(t(e_1))
\overset{e_1^{-1}}{\longrightarrow} \Gamma(e_1)
\overset{\bar{e}_1}{\longrightarrow} \Gamma(o(e_1)) = \Gamma(t'(e)),
\quad \text {for }e\in E(X')\text { with } t(e)=t(e_1), 
\end{equation}
leaving incidence and embeddings unchanged wherever possible. Clearly,
$S'$, the result of contracting the geometric edge $\{e_1,
\bar{e}_1\}$ and deleting the vertex $t(e_1)$, is still a spanning
tree  for $X'$ and, if $(\Gamma(-),X)$ has property $(F_1)$ or
$(F_2)$, then so does $(\Gamma'(-), X')$ by construction. 

It remains to see that the fundamental group of the new graph of
groups $(\Gamma'(-), X')$ is isomorphic to $\Gamma$. The fundamental
group  
\[
\pi_1(\Gamma(-),X,S) 
\]
of the graph of groups $(\Gamma(-),X)$ at the 
spanning tree $S$ is
generated by the groups $\Gamma(v)$ for $v\in V(X)$ plus extra
generators $\gamma_e$ for $e\in \mathcal{O}(X)-E(S)$, where
$\mathcal{O}(X)$ is any orientation of $X$, subject to the relations 
\begin{align}
a^e &= a^{\bar{e}},\quad \text {for }e\in \mathcal{O}(S)\text { and } a\in
\Gamma(e),\label{Eq:GammaPres1}\\[1mm] 
\gamma_e a^e \gamma_e^{-1} &= a^{\bar{e}},\quad \text {for }e\in
\mathcal{O}(X)-E(S)\text { and } a\in \Gamma(e),\label{Eq:GammaPres2} 
\end{align}
where $\mathcal{O}(S)$ is the orientation of the tree $S$ induced by
$\mathcal{O}(X)$, with a corresponding presentation for
$\pi_1(\Gamma'(-),X',S')$;  
see \S5.1 in \cite[Chap.~I]{Serre2}.  The
relations~\eqref{Eq:GammaPres1} corresponding to the geometric edge
$\{e_1, \bar{e}_1\}$ identify $\Gamma(t(e_1))$ isomorphically with a
subgroup of $\Gamma(o(e_1))$; we can thus delete the generators
$\gamma\in \Gamma(t(e_1))$ against those relations by Tietze
moves. This yields a presentation for $\pi_1(\Gamma(-), X,S)$ with the
same set of generators as $\pi_1(\Gamma'(-), X',S')$. Moreover, those
relations \eqref{Eq:GammaPres1}--\eqref{Eq:GammaPres2} coming from
edges $e$ with $t(e)=t(e_1)$ have to be re-expressed in terms of
elements of $\Gamma(o(e_1))$, which leads exactly to the corresponding
relations of $\pi_1(\Gamma'(-), X',S')$ obtained by extending the
embedding $e: \Gamma(e) \rightarrow \Gamma(t(e_1))$ in the natural way
as given in \eqref{Eq:ImbedExt}. Hence, $\pi_1(\Gamma(-), X,S) \cong
\pi_1(\Gamma'(-), X', S')$. Since $V(X)$ is finite, the tree $S$ is
finite; thus, proceeding in the manner described, we obtain, after
finitely many steps, a graph of groups $(\Delta(-), Y)$ with
fundamental group $\Gamma$ and a spanning tree $T$ in $Y$ without
trivial edges, such that $(\Delta(-),Y)$ enjoys the finiteness
properties $(F_1), (F_2)$ whenever $(\Gamma(-),X)$ does.  
\end{proof}

\section{A graph-theoretic lemma}
\label{Sec:TreeOrient}

\noindent The following auxiliary result, which is of an entirely 
graph-theoretic nature, will be used frequently in 
the next two sections.

\begin{lemma}
\label{Lem:TreeOrient}
Let $T$ be a tree, and let $v_0\in V(T)$ be any vertex. Then there
exists one, and only one, orientation $\mathcal{O}(T)$ of $T,$ such
that the assignment $e\mapsto t(e)$ defines a bijection $\psi_{v_0}:
\mathcal{O}(T) \rightarrow V(T)\setminus\{v_0\}$. This orientation is
obtained by orienting each geometric edge so as to point away from the
root $v_0;$ that is, travelling along an oriented edge, the distance
from $v_0$ in the path metric always  increases.  
\end{lemma}

Lemma~\ref{Lem:TreeOrient} is easy to show, even in this
generality. Moreover, for our present purposes, the trees considered
will all be finite, in which case the assertion of
Lemma~\ref{Lem:TreeOrient} may be proved by a straightforward
induction on $\vert V(T)\vert$, which we sketch briefly: by our
condition on the map $\psi_{v_0}$, all (geometric) edges incident with
$v_0$ will have to be oriented away from the root $v_0$.  Delete $v_0$
together with edges incident to $v_0$. The result is a disjoint union
of finitely many subtrees, in which we choose the (previous)
neighbours of $v_0$ as new roots. An application of the induction
hypothesis to these rooted subtrees now finishes the proof.  

\medskip

In what follows, the orientation of a tree $T$ with respect to a base
point $v_0$ described in Lemma~\ref{Lem:TreeOrient} will be denoted by
$\mathcal{O}_{v_0}(T)$.

\section{Classifying virtually free groups of small rank}
\label{sec:class}
\noindent Let $\Gamma$ be a finitely generated virtually free group, and let
$\mu(\Gamma)$ be its free rank, as defined in
Section~\ref{Sec:Prelim}. Then $\Gamma$ is finite if $\mu(\Gamma)=0$,
virtually infinite-cyclic 
if $\mu(\Gamma)=1$, and large in the sense of Pride's preorder on
groups if $\mu(\Gamma)\geq2$. Virtually infinite-cyclic groups play a
certain role in topology as they are precisely the finitely generated
groups with two ends. Their structure is well-known;
cf.\ \cite[5.1]{Stallings} or \cite[Lemma~4.1]{Wall}. Here, we shall
give a short proof of the corresponding result based on the tools
developed in Sections~\ref{Sec:GoGNorm} and \ref{Sec:TreeOrient}.  

\begin{proposition}
\label{Prop:InfCyc}
A virtually infinite-cyclic group $\Gamma$ falls into one of the
following two classes: 
\begin{enumerate}
\item[(i)] $\Gamma$ has a finite normal subgroup with infinite-cyclic quotient.
\vspace{2mm}

\item[(ii)] $\Gamma$ is a free product $\Gamma = G_1
\underset{A}{\ast} G_2$ of two finite groups $G_1$ and $G_2$, with an
amalgamated subgroup $A$ of index $2$ in both factors.  
\end{enumerate}
\end{proposition}

\begin{remark}
\label{Rem:InfCycii}
In Part~(ii) of Proposition~\ref{Prop:InfCyc}, $A$ is a
finite normal subgroup of $\Gamma$ with quotient the infinite dihedral
group $C_2\ast C_2$. 
\end{remark}
\begin{proof}[Proof of Proposition~{\em \ref{Prop:InfCyc}}]
Let $(\Gamma(-),X)$ be a finite graph of finite groups with
fundamental group $\Gamma$ and spanning tree $T$, chosen according to
Lemma~\ref{Lem:Normalise}. The reader should observe that the
assumption that $\Gamma$ is 
virtually infinite-cyclic in combination with \eqref{Eq:FreeEuler} implies that
$\chi(\Gamma)=0$.  

If $\vert V(X)\vert = 1$, $V(X)=\{v\}$ say, then
the above observation together with 
Formula~\eqref{Eq:EulerCharDecomp} show that $X$ has exactly one
geometric edge $\{e, \bar{e}\}$, and that the associated embedding $e:
\Gamma(e)\rightarrow \Gamma(v)$ is an isomorphism. Hence, $\Gamma(v)
\unlhd \Gamma$ and $\Gamma/\Gamma(v) \cong C_\infty$, which gives the
desired result in case (i). 

If $\vert V(X)\vert >1$, we choose an edge $e_1\in E(T)$, introduce
the orientation $\mathcal{O}_{v_0}(T)$ with respect to the base point
$v_0=o(e_1)$, extend it to an orientation $\mathcal{O}(X)$ of $X$,
and let $v_1=t(e_1)$. We then split the Euler characteristic of $\Gamma$ 
as follows: 
\begin{multline}
\label{Eq:chiSplit}
0 = \chi(\Gamma)  
= \underset{v\neq v_0, v_1}{\sum_{v\in V(X)}} \frac{1}{\vert
\Gamma(v)\vert}\,-\,\underset{e\neq
e_1}{\sum_{e\in\mathcal{O}_{v_0}(T)}} \frac{1}{\vert
\Gamma(e)\vert}\,\,+\,\,\Big(\frac{1}{\vert \Gamma(v_0)\vert} +
\frac{1}{\vert \Gamma(v_1)\vert} - \frac{1}{\vert
\Gamma(e_1)\vert}\Big) \\[1mm] 
-\sum_{e\in\mathcal{O}(X)\setminus \mathcal{O}_{v_0}(T)}
\frac{1}{\vert\Gamma(e)\vert}. 
\end{multline}
By the normalisation condition 
\eqref{eq:normal} 
on $(\Gamma(-), X)$, we have 
\[
2 \vert \Gamma(e_1)\vert\, \leq \gamma:=
\min\big\{\vert\Gamma(v_0)\vert, \,\Gamma(v_1)\vert\big\}, 
\]
so
\begin{equation}
\label{Eq:chi=0First}
\frac{1}{\vert \Gamma(v_0)\vert}\, +\,
\frac{1}{\vert\Gamma(v_1)\vert}\, -\, \frac{1}{\vert
\Gamma(e_1)\vert}\, \leq\, \frac{2}{\gamma}\, -\, \frac{1}{\vert
\Gamma(e_1)\vert}\, \leq\, 0. 
\end{equation}
Clearly, equality in \eqref{Eq:chi=0First} occurs if, and only if,
$\Gamma(e_1)$ is of index $2$ in both $\Gamma(v_0)$ and
$\Gamma(v_1)$. 
Similarly, by the normalisation condition \eqref{eq:normal}
and Lemma~\ref{Lem:TreeOrient},
\[
\underset{v\neq v_0, v_1}{\sum_{v\in V(X)}} \frac{1}{\vert
\Gamma(v)\vert}\,-\,\underset{e\neq
e_1}{\sum_{e\in\mathcal{O}_{v_0}(T)}} \frac{1}{\vert
\Gamma(e)\vert}\,=\, \underset{e\neq
e_1}{\sum_{e\in\mathcal{O}_{v_0}(T)}}
\Big(\frac{1}{\vert\Gamma(t(e))\vert} -
\frac{1}{\vert\Gamma(e)\vert}\Big) \leq 0, 
\]
with equality if, and only if, $\mathcal{O}_{v_0}(T) = \{e_1\}$. Also,
trivially, the last sum on the right-hand side of \eqref{Eq:chiSplit}
is non-negative, and vanishes if, and only if,  
 $\mathcal{O}(X)=\mathcal{O}_{v_0}(T)$. Given this discussion, we
conclude from \eqref{Eq:chiSplit} that $\Gamma = \Gamma(v_0)
\underset{\Gamma(e_1)}{\ast} \Gamma(v_1)$, the amalgam being formed
with respect to the embeddings $e_1: \Gamma(e_1) \rightarrow
\Gamma(v_1)$ and $\bar{e}_1: \Gamma(e_1)\rightarrow \Gamma(v_0)$, and
that $(\Gamma(v_0):\Gamma(e_1)^{\bar{e}_1}) = 2 =
(\Gamma(v_1):\Gamma(e_1)^{e_1})$, whence the result in case (ii).  
\end{proof}

\begin{corollary}
\label{Cor:mu=1}
If $\Gamma$ is  virtually infinite-cyclic, then the function
$f_\lambda(\Gamma)$ is constant. More precisely, we have
$f_\lambda(\Gamma)=m_\Gamma$ for $\lambda\geq1$ in   Case~{\em (i)} of
Proposition~{\em \ref{Prop:InfCyc},} while in Case~{\em (ii)} we have
$f_\lambda(\Gamma) = \vert A\vert = m_\Gamma/2$. 
\end{corollary}
\begin{proof}
If $\Gamma$ is as described in Case~(i) of
Proposition~\ref{Prop:InfCyc}, then \eqref{Eq:gExplicit} shows that
$g_\lambda(\Gamma)=1$ for $\lambda\geq0$, leading to
$f_\lambda(\Gamma)=m_\Gamma$ for all $\lambda\geq	1$ by
\eqref{Eq:Transform} and an immediate induction on $\lambda$. 

For $\Gamma$ as in Case~(ii), Equation~\eqref{Eq:gExplicit} yields 
\[
g_\lambda(\Gamma) = 2^{-2\lambda} \binom{2\lambda}{\lambda},\quad \lambda\geq0.
\]
By the binomial theorem applied to the generating function 
$G_\Gamma(z)$ of the $g_\lambda(\Gamma)$'s, we obtain 
$G_\Gamma(z)=(1-z)^{-1/2}$,
which transforms into
the relation
\[
F_\Gamma(z) 
= \frac{\vert m_\Gamma\vert}{2(1-z)}
= \frac{\vert A\vert}{1-z}
\]
via \eqref{Eq:GFTransform}. The desired result follows from this last
equation by  comparing coefficients. 
\end{proof}

By an argument similar to that in the proof of
Proposition~\ref{Prop:InfCyc}, again based on 
Lemmas~\ref{Lem:Normalise} and \ref{Lem:TreeOrient}, one also obtains a
structural classification of the virtually free groups $\Gamma$ of
free rank $\mu(\Gamma)=2$. We confine ourselves to stating the result,
leaving details of the (straightforward if somewhat cumbersome) proof
to the interested reader. 

\begin{proposition}
\label{Prop:mu=2}
A virtually free group $\Gamma$ of rank $\mu(\Gamma)=2$ falls into one
of the following five classes: 
\begin{enumerate}
\item[(i)] $\Gamma$ is an $HNN$-extension $\Gamma =
G\underset{S,\sigma}{\ast}$ with finite base group $G$ and $(G:S)=2$. 
\vspace{2mm}

\item[(ii)] $\Gamma$ contains a finite normal subgroup $G$ with
quotient $\Gamma/G \cong F_2$ free of rank $2$. 
\vspace{2mm}

\item[(iii)] $\Gamma$ is a free product $\Gamma = G_1
\underset{S}{\ast} G_2$ of two finite groups $G_i$ with an amalgamated
subgroup $S,$ whose indices $(G_i:S)$ satisfy one of the conditions 
\vspace{1mm}

\begin{enumerate}
\item[(iii)$_1$] $\{(G_1:S),\,(G_2:S)\} = \{2,3\},$
\vspace{1mm}

\item[(iii)$_2$] $(G_1:S) = 3 = (G_2:S),$
\vspace{1mm}
\item[(iii)$_3$] $\{(G_1:S),\, (G_2:S)\} = \{2, 4\}$.
\end{enumerate}
\vspace{2mm}

\item[(iv)] $\Gamma$ is of the form $\Gamma = (G_1
\underset{G_{12}}{\ast} G_2) \underset{G_{23}}{\ast} G_3$ with finite
factors $G_i$ and indices\\ $(G_1:G_{12}) = 2 = (G_2:G_{12})$ and
$(G_2:G_{23}) = 2 = (G_3: G_{23})$. 
\vspace{2mm}

\item[(v)] $\Gamma$ is a free product $\Gamma = G_1 \underset{S}{\ast}
G_2,$ where $G_1$ contains a finite normal subgroup $H_1$ with $S\leq
H_1$ and $G_1/H_1 \cong C_\infty,$ $G_2$ is finite, and $(H_1:S) = 2 =
(G_2:S)$. 
\end{enumerate}
\end{proposition} 

\section{Some criteria for a virtually free group to be `large'}
\label{Sec:Large}

\noindent Our next result collects together a number of equivalent conditions on
a finitely generated virtually free group $\Gamma$ which all say, in
one way or another, that $\Gamma$ is `large' in some particular
sense. Perhaps the most obvious condition in this direction is given
by Pride's concept of being `as large as a free group of rank
$2$'. The concept of `largeness' for groups, first introduced by
S. Pride in \cite{Pride}, and further developed in \cite{EP}, depends
on a certain preorder $\preceq$ on the class of groups, defined in
\cite{EP} as follows: let $G$ and $H$ be groups. Then we write
$H\preceq G$, if there exist 
\begin{enumerate}
\item[(a)] a subgroup $G^0$ of finite index in $G$;
\vspace{2mm}

\item[(b)] a subgroup $H^0$ of finite index in $H$, and a finite
normal subgroup $N^0$ of $H^0$; 
\vspace{2mm}

\item[(c)] a homomorphism from $G^0$ onto $H^0/N^0$.
\end{enumerate}

We write $H\sim G$ if $H\preceq G$ and $G\preceq H$, and we denote by
$[G]$ the equivalence class of the group $G$ under $\sim$. By abuse of
notation, we also denote by $\preceq$ the preorder induced on the
class of equivalence classes of groups. The finitely generated groups
which are `largest' in Pride's sense are the ones having a subgroup of
finite index which can be mapped homomorphically onto the free group
$F_2$ of rank $2$. 

Another, more topological, way of saying that a finitely generated
virtually free group is `large', is that it has infinitely many
ends. Here, the number $e(\Gamma)$ of ends of a group $\Gamma$ is
defined as 
\[
e(\Gamma) = \begin{cases}
\dim H^0(\Gamma, \mathrm{Hom}_\Z(\Z\Gamma, \Z_2)/\Z_2\Gamma),&
\mbox{if $\Gamma$ is infinite,}\\[2mm] 
0,& \mbox{if $\Gamma$ is finite.}
\end{cases}
\]
The reader is referred to \cite{Cohen1} or \cite[Sec.~2]{Cohen2} for
an  introduction to the theory of ends of a group from an algebraic
point of view; for a discussion from a more topological viewpoint,
see, for instance, \cite{Hopf}, \cite{Freud}, or \cite{Specker}.  

Criterion (vii) below will play a role in the proof of
Lemma~\ref{Lem:Main}, which is the main tool in establishing
Theorem~\ref{Thm:Periodic}. 
In item~(vi), the symbol $s_m(\Gamma)$ denotes the number of
subgroups of index $m$ in $\Gamma$.

\begin{lemma}
\label{Lem:chi<0Equivs}
Let $\Gamma$ be a finitely generated virtually free group, and let
$(\Gamma(-), X)$ be a finite graph of finite groups with fundamental
group $\Gamma,$ chosen so as to satisfy the normalisation condition 
\eqref{eq:normal} of
Lemma~{\em \ref{Lem:Normalise}}. Then the following assertions on
$\Gamma$ are equivalent: 
\begin{enumerate}
\item[(i)] $\chi(\Gamma) < 0$.
\vspace{2mm}

\item[(ii)] $\mu(\Gamma) \geq2$.
\vspace{2mm}

\item[(iii)] $\Gamma$ has infinitely many ends.
\vspace{2mm}

\item[(iv)] The function $f_\lambda(\Gamma)$ is strictly increasing.
\vspace{2mm}

\item[(v)] $\Gamma\sim F_2$ in the sense of Pride's preorder $\preceq$
on groups, where $F_2$ denotes the free group of rank $2$. 
\vspace{2mm}

\item[(vi)] $\Gamma$ has fast subgroup growth in the sense that the
inequality $s_{nj}(\Gamma) \geq c\cdot n!$ holds for some fixed positive
integer $j,$  some constant $c>0,$ and all $n\geq1$. 
\vspace{2mm}

\item[(vii)] If $X$ has only one vertex $v,$ then either $X$ has more
than one geometric edge, or $E(X) = \{e_1, \bar{e}_1\}$ and
$(\Gamma(v): \Gamma(e_1)^{e_1}) \geq2;$ if $\vert V(X)\vert \geq2,$
then $X$ is not a tree, or $X$ is a tree with more than one geometric
edge, or $E(X) = \{e_1, \bar{e}_1\}$ and $\chi(\Gamma_0)<0,$ where
$\Gamma_0:= \Gamma_{o(e_1)} \underset{\Gamma(e_1)}{\ast}
\Gamma_{t(e_1)}$. 
\end{enumerate}
\end{lemma}

\begin{proof} 
(i) $\Leftrightarrow$ (ii). This is immediate from
Formula~\eqref{Eq:FreeEuler} plus the fact that $\mu(\Gamma)$ is
integral. 

(ii) $\Leftrightarrow$ (iii). This follows from
\cite[Prop.~2.1]{Cohen2} (i.e., the fact that the number of ends is
invariant when passing to a subgroup of finite index) and Examples~1
and 2 in \cite{Cohen2} computing the number of ends of a free product,
respectively of $C_\infty$.  

(ii) $\Leftrightarrow$ (iv). This follows from
\cite[Theorem~4]{MuDiss} in conjunction with
Corollary~\ref{Cor:mu=1}. 

(ii) $\Rightarrow$ (v). If $\mu(\Gamma)\geq2$, then $\Gamma$ contains
a free group $F$ of rank at least $2$, with $(\Gamma:F) =
m_\Gamma<\infty$; in particular, $F_2 \preceq \Gamma$. 
Since $[F_2]$ is largest with
respect to the preorder $\preceq$ among all equivalence classes of
finitely generated groups, we also have $\Gamma \preceq F_2$, so
$\Gamma \sim F_2$, as claimed. 

(v) $\Rightarrow$ (vi). Suppose that $\Gamma\sim F_2$. Then there
exists a subgroup $\Delta\leq \Gamma$ of index
$(\Gamma:\Delta)=j<\infty$ and a surjective homomorphism $\varphi:
\Delta\rightarrow F_2$. From this plus Newman's asymptotic estimate
\cite[Theorem~2]{Newman} 
\[
s_n(F_r) \sim n (n!)^{r-1}\mbox{ as }n\rightarrow\infty,\quad r\geq2,
\]
it follows that
\[
s_{jn}(\Gamma) \geq s_n(\Delta) \geq s_n(F_2) \geq c\cdot n\cdot n! \geq
c\cdot n!
\]
for $n\geq1$ and some constant $c>0$, whence (vi).

(vi) $\Rightarrow$ (ii). If $\mu(\Gamma)\leq 1$, then either $\Gamma$
is finite, so $s_n(\Gamma)=0$ for sufficiently large $n$, or $\Gamma$
is virtually infinite-cyclic, implying 
\[
s_n(\Gamma) \leq n^\alpha,\quad n\geq1,
\]
for some constant $\alpha$, by \cite[Cor.~1.4.3]{LS}; see also
\cite{Segal}. In both cases, Condition~(vi) does not hold. 

(ii) $\Leftrightarrow$ (vii). This follows by splitting the Euler
characteristic $\chi(\Gamma)$ as in the proof of
Proposition~\ref{Prop:InfCyc}, making use of
Lemmas~\ref{Lem:Normalise} and ~\ref{Lem:TreeOrient}. 
\end{proof}

\section{The main lemma}
\label{sec:mainlemma}

\noindent For a positive integer $n$ and a prime number $p$, denote by
$v_p(n)$ the $p$-adic valuation of $n$, that is, the exponent of the
highest $p$-power dividing $n$. 

\begin{lemma}
\label{Lem:Main}
Let $\Gamma$ be a finitely generated virtually free group of rank
$\mu(\Gamma)\geq2,$ and let $p$ be a prime number not dividing
$m_\Gamma$. Then we have  
\begin{equation}
\label{Eq:nupGenEst}
v_p\big(g_\lambda(\Gamma)\big) \geq v_p(\lambda!),\quad \lambda\geq0.
\end{equation}
In particular, the function $v_p\big(g_\lambda(\Gamma)\big)$ is
non-negative, and unbounded as $\lambda\rightarrow\infty$. 
\end{lemma}

\begin{proof}
Let $(\Gamma(-), X)$ be a Stallings decomposition of $\Gamma$, chosen
according to Lemma~\ref{Lem:Normalise}, and let $\mathcal{O}(X)$ be
any orientation of the graph $X$. Then, by
Formula~\eqref{Eq:gExplicit}, plus the fact that $p\nmid m_\Gamma$, we
have   
\begin{equation}
\label{Eq:NormComp}
v_p\big(g_\lambda(\Gamma)\big) = \sum_{e\in\mathcal{O}(X)} v_p\big((\lambda
m_\Gamma/\vert\Gamma(e)\vert)!\big) \, - \, \sum_{v\in V(X)}
v_p\big((\lambda m_\Gamma/\vert\Gamma(v)\vert)!\big),\quad \lambda\geq0. 
\end{equation}

We now distinguish two cases depending on whether $\vert V(X)\vert=1$
or $\vert V(X)\vert\geq2$. 

(a) $V(X) = \{v\}$. Let $\mathcal{O}(X) = \{e_1,\ldots, e_r\}$. Then
$m_\Gamma = \vert \Gamma(v)\vert$, and \eqref{Eq:NormComp} becomes 
\[
v_p\big(g_\lambda(\Gamma)\big) = \sum_{\rho=1}^r v_p\big((\lambda \vert
\Gamma(v)\vert/\vert\Gamma(e_\rho)\vert)!\big)\, -\, v_p(\lambda!),\quad
\lambda\geq0. 
\]
Since $\mu(\Gamma)\geq2$ by hypothesis, the implication (ii)
$\Rightarrow$ (vii) of Lemma~\ref{Lem:chi<0Equivs} tells us that
either $r\geq2$, or $r=1$ and $2\vert\Gamma(e_1)\vert \leq
\vert\Gamma(v)\vert$. Since $\vert \Gamma(e_\rho)\vert \leq \vert
\Gamma(v)\vert$, we conclude that  
\[
v_p\big(g_\lambda(\Gamma)\big) \,\geq\,\left.\begin{cases}
v_p(\lambda!),&r\geq2,\\[2mm]
v_p\left(\binom{2\lambda}{\lambda}\right) + v_p(\lambda!),& r=1,
\end{cases}\right\} \,\geq\, v_p(\lambda!),\quad \lambda\geq0.
\]

(b) $\vert V(X)\vert \geq2$. Let $T$ be a spanning tree in $X$, let
$e_1\in E(T)$ be any edge, and let $\mathcal{O}(X)$ be an orientation
of $X$ extending $\mathcal{O}_{v_0}(T)$, where $v_0 = o(e_1)$. Let
$v_1 = t(e_1)$ and let $\Gamma_0$ be as in
Lemma~\ref{Lem:chi<0Equivs}(vii). Rewriting \eqref{Eq:NormComp} by
means of Legendre's formula for the $p$-part of factorials, and
estimating resulting quantities by means of the inequality 
\[
\lfloor x + y \rfloor \geq \lfloor x\rfloor + \lfloor y\rfloor,\quad
\text {for }x,y\in\mathbb{R}, 
\]
we get
\allowdisplaybreaks
\begin{align}
\label{Eq:pPart}
v_p\big(g_\lambda(\Gamma)\big) &= v_p\big((\lambda m_\Gamma/\vert
\Gamma(e_1)\vert)!\big) - v_p\big((\lambda m_\Gamma/\vert
\Gamma(v_0)\vert)!\big) - v_p\big((\lambda m_\Gamma/\vert
\Gamma(v_1)\vert)!\big)\notag\\[1mm]   
&\hspace{2cm}+\,\sum_{e\in\mathcal{O}_{v_0}(T)\setminus\{e_1\}}
\Big(v_p\big((\lambda
m_\Gamma/\vert \Gamma(e)\vert)!\big) - v_p\big((\lambda m_\Gamma/\vert 
\Gamma(t(e))\vert)!\big)\Big)\notag\\[1mm]     
&\hspace{2cm}+\, \sum_{e\in\mathcal{O}(X)\setminus
\mathcal{O}_{v_0}(T)} v_p\big((\lambda m_\Gamma/\vert
\Gamma(e)\vert)!\big)\notag\\[2mm]  
&= \sum_{\mu\geq1}\Bigg(\Big\lfloor \frac{\lambda
m_\Gamma/\vert\Gamma(e_1)\vert}{p^\mu}\Big\rfloor - \Big\lfloor
\frac{\lambda m_\Gamma/\vert\Gamma(v_0)\vert}{p^\mu}\Big\rfloor -
\Big\lfloor \frac{\lambda
m_\Gamma/\vert\Gamma(v_1)\vert}{p^\mu}\Big\rfloor\notag\\[1mm]  
&\hspace{2cm} +
\sum_{e\in\mathcal{O}_{v_0}(T)\setminus\{e_1\}}\bigg(\Big\lfloor
\frac{\lambda m_\Gamma/\vert\Gamma(e)\vert}{p^\mu}\Big\rfloor -
\Big\lfloor \frac{\lambda
m_\Gamma/\vert\Gamma(t(e))\vert}{p^\mu}\Big\rfloor\bigg)\notag\\[1mm] 
&\hspace{2cm}+ \sum_{e\in\mathcal{O}(X)\setminus\mathcal{O}_{v_0}(T)}
\Big\lfloor \frac{\lambda
m_\Gamma/\vert\Gamma(e)\vert}{p^\mu}\Big\rfloor\Bigg)\notag\\[2mm] 
&\geq\sum_{\mu\geq1}\Bigg(\Big\lfloor \frac{-\lambda m_\Gamma
\chi(\Gamma_0)}{p^\mu}\Big\rfloor +
\sum_{e\in\mathcal{O}_{v_0}(T)\setminus\{e_1\}} \left\lfloor
\frac{\lambda m_\Gamma}{p^\mu}\Big(\frac{1}{\vert\Gamma(e)\vert} -
\frac{1}{\vert\Gamma(t(e))\vert}\Big)\right\rfloor\notag\\[1mm]
&\hspace{5cm}+ \sum_{e\in \mathcal{O}(X)\setminus\mathcal{O}_{v_0}(T)}
\Big\lfloor \frac{\lambda
m_\Gamma/\vert\Gamma(e)\vert}{p^\mu}\Big\rfloor\Bigg) .
\end{align} 
By the normalisation condition \eqref{eq:normal} on $(\Gamma(-), X)$, we have
$\chi(\Gamma_0)\leq 0$ and 
\[
2\vert\Gamma(e)\vert \leq \vert \Gamma(t(e))\vert,\quad
\text {for }e\in\mathcal{O}_{v_0}(T)\setminus\{e_1\}, 
\]
so that the inequality for $v_p\big(g_\lambda(\Gamma)\big)$ resulting from
\eqref{Eq:pPart} entails 
\begin{multline}
\label{Eq:pPartEst}
v_p\big(g_\lambda(\Gamma)\big) \geq v_p\big((\lambda m_\Gamma
m_{\Gamma_0}^{-1}(\mu(\Gamma_0)-1))!\big) +
\sum_{e\in\mathcal{O}_{v_0}(T)\setminus\{e_1\}} v_p\big((\lambda
m_\Gamma/\vert\Gamma(t(e))\vert)!\big)\\[2mm]  
+ \sum_{e\in\mathcal{O}(X)\setminus\mathcal{O}_{v_0}(T)}
v_p\big((\lambda m_\Gamma/\vert\Gamma(e)\vert)!\big). 
\end{multline}
By hypothesis, we have $\mu(\Gamma)\geq2$ hence, again by the
implication (ii) $\Rightarrow$ (vii) of  Lemma~\ref{Lem:chi<0Equivs},
at least one of the assertions 
\begin{enumerate}
\item[(1)] $\mathcal{O}(X)\setminus\mathcal{O}_{v_0}(T) \neq \emptyset$,
\vspace{2mm}
\item[(2)] $\vert\mathcal{O}_{v_0}(T)\vert \geq2$,
\vspace{2mm}

\item[(3)] $\mu(\Gamma_0)\geq2$
\end{enumerate}
holds true. It follows now from \eqref{Eq:pPartEst} that the estimate 
\eqref{Eq:nupGenEst} also holds in Case~(b), finishing the proof of 
the lemma.
\end{proof}

\begin{remark}
\label{Rem:nupGenEst}
It is easy to see that the estimate \eqref{Eq:nupGenEst} in
Lemma~\ref{Lem:Main} is, in general, best possible. For
instance, let $\Gamma= \PSL_2(\Z) \cong C_2\ast C_3$. Then 
\[
g_\lambda(\Gamma) = \frac{(6\lambda)!}{(2\lambda)!\, (3\lambda)!\,
2^{3\lambda}\, 3^{2\lambda}}, 
\]
and, for a prime $p\geq7$ and $\lambda=p^r$ with $r\geq0,$ we have
\[
v_p\big(g_\lambda(\Gamma)\big) = \frac{p^r-1}{p-1} = v_p(\lambda!).
\]
\end{remark}

\section{Proof of Theorem~\ref{Thm:Periodic}}
\label{Sec:PeriodicThmProof}

\noindent Theorem~\ref{Thm:Periodic} follows easily from
Lemmas~\ref{Lem:Periodicity} and \ref{Lem:Main}. Indeed, if
$\mu(\Gamma)=0$, then $\Gamma$ is finite, we have
$m_\Gamma=\vert\Gamma\vert$, and thus 
$f_1(\Gamma)=1$ and $f_\lambda(\Gamma)=0$ for $\lambda\ge2$, 
so that $f_\lambda(\Gamma)$ is ultimately periodic with period and
preperiod equal to $1$ modulo any prime power. If $\mu(\Gamma) = 1$,
then, by Corollary~\ref{Cor:mu=1}, $f_\lambda(\Gamma)$ is constant,
thus purely periodic with period equal to $1$, again modulo any prime
power. Now suppose that $\mu(\Gamma)\geq2$. Given a positive integer
$\alpha$, let $\lambda_0(\alpha)$ be chosen according to
Lemma~\ref{Lem:Main} such that $v_p\big(g_\lambda(\Gamma)\big) \geq\alpha$
for all $\lambda\geq \lambda_0(\alpha)$ and
$v_p\big(g_{\lambda_0(\alpha)-1}(\Gamma)\big) < \alpha$.  
Then consider Equation~\eqref{Eq:Transform} for
$\lambda\geq\lambda_0(\alpha)$. All summands on the left-hand side
corresponding to indices $\mu\geq \lambda_0(\alpha)$ will vanish
modulo~$p^\alpha$, as does the right-hand side, and we obtain the
congruence 
\begin{equation}
\label{Eq:p^alphaCong}
f_{\lambda+\lambda_0(\alpha)}(\Gamma) \equiv -\big(g_1(\Gamma)
f_{\lambda+\lambda_0(\alpha)-1}(\Gamma) + \cdots +
g_{\lambda_0(\alpha)-1}(\Gamma) f_{\lambda+1}(\Gamma)\big)
\pmod{p^\alpha},\quad 
\lambda\geq0.  
\end{equation}
Applying Lemma~\ref{Lem:Periodicity} with $\Lambda=\Z/p^\alpha\Z$ and
$\mathcal{S} = (f_{\lambda+1}(\Gamma))_{\lambda\geq0}$, we find that,
in this last case, $f_\lambda(\Gamma)$ is ultimately periodic 
modulo~$p^\alpha$ with least period $\omega_0 <
p^{\alpha(\lambda_0(\alpha)-1)}$, whence the result. 

\section{Proof of Theorem~\ref{Thm:mup>0}}
\label{Sec:prop}

\noindent If $p\mid m_\Gamma$, then, by
\cite[Eq.~(3)]{MuPuFree}, the generating function $F_\Gamma(z)$
satisfies the congruence
\begin{equation} \label{eq:Fmodp} 
F_\Gamma(z)=z^{\mu_p(\Gamma)}F_\Gamma^{\mu_p(\Gamma)}(z)
\left(z^{p-1}F_\Gamma(z)^{p-1}-1\right)^{(\mu(\Gamma)-\mu_p(\Gamma))/(p-1)}
\quad \text{modulo }p.
\end{equation}
As is argued in \cite{MuPuFree}, if $\mu_p(\Gamma)>0$,
then it is obvious from this congruence
that $F_\Gamma(z)=0$~modulo~$p$. 

We shall now demonstrate by an induction on $\alpha$ that,
for all integers $\alpha\ge1$, the generating function
$F_\Gamma(z)$ is rational when coefficients are reduced
modulo~$p^\alpha$. For $\alpha=1$ this is true due
to the above remark.

Let us now suppose that we have already shown that
$F_\Gamma(z)$ is rational when coefficients are reduced
modulo~$p^\alpha$, say $F_\Gamma(z)=R(z)$~modulo~$p^\alpha$,
for some rational function $R(z)$ over the integers
whose denominator is not divisible by $p$. 
By \cite[Eq.~(12)]{MuPuFree} and \eqref{eq:Fmodp}, we know that
\begin{multline} \label{eq:FGl} 
F_\Gamma(z)=z^{\mu_p(\Gamma)}F_\Gamma^{\mu_p(\Gamma)}(z)
\left(z^{p-1}F_\Gamma(z)^{p-1}-1\right)^{(\mu(\Gamma)-\mu_p(\Gamma))/(p-1)}\\
+p\cdot\mathcal P(z,F_\Gamma(z),F'_\Gamma(z),
F''_\Gamma(z),\dots,F^{(\mu(\Gamma-1))}_\Gamma(z)),
\end{multline}
where $\mathcal P(z,F_\Gamma(z),F'_\Gamma(z),\dots,
F^{(\mu(\Gamma)-1)}_\Gamma(z))$ is
a polynomial in $z,F_\Gamma(z),F'_\Gamma(z),\dots,\break
F^{(\mu(\Gamma)-1)}_\Gamma(z)$
over the {\it rationals}. However, it is proven in \cite[Sections~3
  and~5]{MuPuFree} that, if $p\mid m_\Gamma$, the rational
coefficients can be written with denominators which are relatively
prime to $p$, a fact that we shall tacitly use in the sequel.

We now make the Ansatz $F_\Gamma(z)=R(z)+p^{\alpha}Y(z)$, for some
unknown formal power series $Y(z)$, we substitute in \eqref{eq:FGl},
and then consider the result modulo~$p^{\alpha+1}$. Since
$$
(R(z)+p^{\alpha}Y(z))^e=R^e(z)+ep^\alpha R^{e-1}(z)Y(z)
\quad \text{modulo }p^{\alpha+1},
$$
we arrive at the congruence
\begin{multline} \label{eq:Fcong} 
R(z)+p^{\alpha}Y(z)=
\sum_{i=0}^M (-1)^{M-i}\binom Mi z^{\mu_p(\Gamma)+i(p-1)}\\
\kern4cm
\cdot
\left(
R^{\mu_p(\Gamma)+i(p-1)}(z)
+p^\alpha(\mu_p(\Gamma)+i(p-1))R^{\mu_p(\Gamma)+i(p-1)-1}(z)Y(z)
\right)
\\
+p\cdot\mathcal P(z,R(z),R'(z),\dots,R^{(\mu(\Gamma)-1)}(z))
\quad \quad \text{modulo }p^{\alpha+1},
\end{multline}
where $M$ is short for $(\mu(\Gamma)-\mu_p(\Gamma))/(p-1)$.
By rearranging terms, this turns into
\begin{multline} \label{eq:Fcong2} 
p^\alpha Y(z)\cdot
\bigg(-1+
\sum_{i=0}^M (-1)^{M-i}\binom Mi z^{\mu_p(\Gamma)+i(p-1)}
(\mu_p(\Gamma)+i(p-1))R^{\mu_p(\Gamma)+i(p-1)-1}(z)
\bigg)
\\
=
R(z)-
\sum_{i=0}^M (-1)^{M-i}\binom Mi z^{\mu_p(\Gamma)+i(p-1)}
R^{\mu_p(\Gamma)+i(p-1)}(z)\\
-
p\cdot\mathcal P(z,R(z),R'(z),\dots,R^{(\mu(\Gamma)-1)}(z))
\quad \quad \text{modulo }p^{\alpha+1}.
\end{multline}
By induction hypothesis, the right-hand side is divisible by $p^\alpha$.
We may hence divide both sides by $p^\alpha$, to obtain the congruence
\begin{multline*} 
Y(z)\cdot
\bigg(-1+
\sum_{i=0}^M (-1)^{M-i}\binom Mi z^{\mu_p(\Gamma)+i(p-1)}
(\mu_p(\Gamma)-i)R^{\mu_p(\Gamma)+i(p-1)-1}(z)
\bigg)
=S(z)\\
\quad \quad \text{modulo }p,
\end{multline*}
where $S(z)$ can be written as 
an explicit rational function in $z$ over the integers
with denominator not divisible by $p$.
If we remember that,
from the base case of the induction (see the sentence below
\eqref{eq:Fmodp}), it follows that $R(z)=0$~modulo~$p$,
then we see that the above congruence simplifies further to
\begin{equation} \label{eq:Ycong} 
Y(z)\cdot
\bigg(-1+
(-1)^{M} \mu_p(\Gamma)z^{\mu_p(\Gamma)}
R^{\mu_p(\Gamma)-1}(z)
\bigg)
=S(z)
\quad \quad \text{modulo }p.
\end{equation}
We can therefore
determine $Y(z)$ modulo~$p$ by dividing both sides 
of the congruence by the term in parentheses on the
left-hand side. 
Hence $Y(z)$ is a rational function modulo~$p$,
and therefore $F_\Gamma(z)=R(z)+p^\alpha Y(z)$ is rational
modulo~$p^{\alpha+1}$. This concludes the induction argument.

However, the additional assertions in Theorem~\ref{Thm:mup>0}
are now also obvious: if $\mu_p(\Gamma)=1$, then every time
we divide by $1-(-1)^Mz$, and the induction hypothesis guarantees
that all denominators of fractions in the congruence \eqref{eq:Ycong} are a 
power of $1-(-1)^Mz$, 
cf.\ the implicit definition of $S(z)$ via \eqref{eq:Fcong2}. 
On the other hand,
if $\mu_p(\Gamma)\ge2$, then as induction hypothesis let us suppose that
$R(z)$ is actually a {\it polynomial\/}
modulo~$p^\alpha$. Again using the fact that $R(z)=0$~modulo~$p$, we see 
that, since $\mu_p(\Gamma)+i(p-1)>0$ for all
$i\ge0$ in the current case, the congruence \eqref{eq:Ycong} reduces to
$$
-Y(z)
=S(z)
\quad \quad \text{modulo }p.
$$
Here, the rational function $S(z)$ is actually a {\it polynomial}.
Consequently, $Y(z)$ is a polynomial modulo~$p$, and thus
$F_\Gamma(z)=R(z)+p^\alpha Y(z)$ is a polynomial
modulo~$p^{\alpha+1}$. This completes the proof of the theorem.

\section{Proof of Theorem~\ref{Thm:Main}}
\label{Sec:Main}

\noindent (i) {\it implies} (ii). This is obvious.

\medskip
(ii) {\it implies} (iii).
We may have $p\nmid m_\Gamma$ (Case~(iii)$_1$), or
$p\mid m_\Gamma$ and $\mu_p(\Gamma)>0$ (Case~(iii)$_2$),
or $p\mid m_\Gamma$ and $\mu_p(\Gamma)=0$.
Due to the definition \eqref{Eq:mupDef} of $\mu_p$, the condition 
$\mu_p(\Gamma)=0$ already implies that $p\mid m_\Gamma$
(in the sum on the right-hand side of \eqref{Eq:mupDef} only
the term for $\kappa=m_\Gamma$ can be negative).
Thus, it remains to consider the case where
$\mu_p(\Gamma)=0$. In this case, Theorem~A(iii) in
\cite{MuPuFree} says that either $\mu(\Gamma)=0$,
or $\mu(\Gamma)=1$ and $p=2$. In the former
case, the group $\Gamma$ is finite (Case~(iii)$_3$), while in the
latter case $\Gamma$ is virtually infinite-cyclic (Case~(iii)$_4$).

\medskip
(iii) {\it implies} (i).
We have to distinguish between the various subcases given
in this item. In each case, we have to show that the
generating function $F_\Gamma(z)$ is rational modulo~$p^\alpha$
for every $\alpha\ge1$.

\smallskip
{\it Case} (iii)$_1$. This is taken care of by Theorem~\ref{Thm:Periodic}.

\smallskip
{\it Case} (iii)$_2$. This is dealt with by Theorem~\ref{Thm:mup>0}.

\smallskip
{\it Case} (iii)$_3$. This is obvious since in this case $F_\Gamma(z)=1$.

\smallskip
{\it Case} (iii)$_4$. Corollary~\ref{Cor:mu=1} says that in this case
the sequence of free subgroup numbers $f_\lambda(\Gamma)$ is
constant. Consequently, the corresponding generating function
$F_\Gamma(z)$ is rational even over the integers.

\end{document}